\documentclass{article}
\usepackage[utf8]{inputenc}
\usepackage{amsthm}
\usepackage{amssymb}
\usepackage{amsmath}
\usepackage{graphicx}
\usepackage{bbm}
\usepackage{hyperref}
\usepackage{comment}
\usepackage[
backend=biber,
style=alphabetic,
sorting=nyt
]{biblatex}
\hypersetup{
    colorlinks=true,
    linkcolor=blue,
    filecolor=magenta,      
    urlcolor=cyan,
    pdftitle={Overleaf Example},
    pdfpagemode=FullScreen,
    }
\usepackage[cmtip,all]{xy}

\newtheorem{theorem}{Theorem}[section]

\newtheorem{lemma}[theorem]{Lemma}
\newtheorem{proposition}[theorem]{Proposition}

\newtheorem*{teo*}{Theorem}

\theoremstyle{definition}

\theoremstyle{remark}
\newtheorem{remark}[theorem]{Remark}

\numberwithin{equation}{section}

\DeclareMathOperator{\N}{\mathbb{N}}
\DeclareMathOperator{\Q}{\mathbb{Q}}

\DeclareMathOperator{\R}{\mathbb{R}}
\DeclareMathOperator{\C}{\mathbb{C}}
\DeclareMathOperator{\Z}{\mathbb{Z}}

\DeclareMathOperator{\NN}{\mathcal{N}}

\DeclareMathOperator{\Mat}{\text{Mat}}
\DeclareMathOperator{\Id}{\text{Id}}
\newcommand{\RM}[1]{\textup{\uppercase\expandafter{\romannumeral#1}}}

\title{On Limit Eigenvalue Distributions Associated to Residual Chains of Groups}
\author{
  Jan Boschheidgen 
  \footnote{Departamento de Matem\'aticas, Universidad Aut\'onoma de Madrid}}

\begin{document}

\maketitle
\begin{abstract}
    \noindent
    Let $G$ be a residually finite group.
    We give an explicit example in the discrete Heisenberg group that the Brown measure of multiplication operators 
    $A \in \Z[G] \subseteq \mathcal{B}(\ell^2(G))$ 
    in general can not be approximated using finite quotients $G/N$ of $G$.
    We show that in finitely generated  abelian groups the Brown measure can be approximated using finite quotients.
\end{abstract}

\section{Introduction}
Let $G$ be a residually finite group and
let $A \in \text{Mat}_n(\mathbb{C}[G])$
be a matrix over the complex group ring. Let
$G \unrhd N_1 \unrhd N_2  \unrhd \ldots$ be a chain of normal subgroups in $G$ of finite index with trivial 
intersection and set $G_i = G / N_i.$ By right 
multiplication we get an action of $G$ on 
$\mathbb{C}[G_i] \cong \mathbb{C}^{\vert G_i \vert}$.
This action extends linearly to $\C[G]$ and induces an action of $\Mat_n(\C[G])$ on $\C[G_i]^n$. Let $A_i \in \text{Mat}_{n\cdot\vert G_i \vert}(\mathbb{C})$
be the matrix associated to the linear operator on $\mathbb{C}[G_i]^n \cong \mathbb{C}^{n\cdot\vert G_i \vert}$ given by right multiplication by $A$ with respect to some basis.
Let $\lambda_1, \ldots, \lambda_{n\cdot \vert G_i \vert}$ be the eigenvalues of $A_i$.
We define the regularized eigenvalue measure of the matrix
$A_i$ as
\begin{equation}\label{group_reg_eig_measure}
    \mu_{A_i} = \frac{1}{\vert G_i \vert}\sum\limits_{j = 1}^{n\cdot \vert G_i \vert} \delta_{\lambda_j}
\end{equation}
where $\delta_z $ is the Dirac measure at $z \in \mathbb{C}.$
Note that in the measure the eigenvalues appear with multiplicities.
We are interested in the following questions:
\begin{itemize}
    \item[(1)] Does the limit 
    $\lim\limits_{i \to \infty} \mu_{A_i}(\{0\})$
    exist?
    \item[(2)] If the answer to question (1) is yes, is the limit independent of the chain $(N_i)_i$?
    \item[(3)] Let $\mu_A$ be the Brown measure of the operator $r_A$ on $(\ell^2(G))^n$ given by right multiplication by $A$. Do the measures $\mu_{A_i}$ converge weakly to $\mu_A$?
\end{itemize}
For the definition of the Brown measure see section \ref{Section_Brown_measure}.

We briefly want to discuss the case when $A \in \Mat_n(\C[G])$ is normal.
If $A$ is a normal matrix, the answer to all of the above questions is yes. In this case the measure $\mu_A$ is known as the \textit{spectral measure}.
By the Weierstrass approximation theorem, 
to show the weak convergence of the measures, it is enough to show
\begin{equation}\label{weak_convergence_normal}
\lim\limits_{i \to \infty}\int\limits_{\C} x^n \overline{x}^m d\mu_{A_i} = \int\limits_{\C} x^n \overline{x}^m d\mu_A.
\end{equation}
We will explain this equality in section \ref{Section_Brown_measure}. 
We can then ask about convergence of $\mu_{A_i}(\{0\}).$
By the theorem of Portmanteau the weak convergence already gives 
$$
\limsup\limits_{i \to \infty} \mu_{A_i}(\{0\}) \leq \mu_A (\{0\}).
$$
The other inequality goes back 
to Wolfgang L{\"u}ck \cite{Luck_start}
and Andrei Jaikin-Zapirain \cite{Jaikin1}.
For positive self adjoint matrices $A \in \Mat_n(\Z[G])$, L{\"u}ck discovered that since the characteristic polynomial of $A_i$ has coefficients in $\Z$, the product of all non zero eigenvalues of $A_i$ has absolute value greater then one. Further the absolute value of each eigenvalue of $A_i$ is bounded by the operator norm of $A$. Thus there can not be too many "small" eigenvalues, since otherwise there are not enough "large" eigenvalues  such that their product is greater then one. In fact he proved
the existence of a constant $B$ that only depends on $A$ such that
\begin{equation}\label{uniform_bound}
\mu_{A_i}(0, \lambda) \leq \frac{B}{|\log \lambda |}
\end{equation}
for all $\lambda \in (0, 1).$

From that, using again the theorem of Portmanteau, one can deduce that
\begin{equation}\label{measure_Equality}
    \lim\limits_{i \to \infty} \mu_{A_i}(\{0\}) = \mu_A(\{0\}).
\end{equation}
It is easy to generalize the result to normal matrices $A \in \Mat_n(\Q[G])$.
Since for any arbitrary matrix $A \in \Mat_n(\C[G])$ we have $\ker r_A = \ker r_A r_A^\ast$ we can deduce from equation \ref{measure_Equality} the so called \textit{L{\"u}ck Approximation} for arbitrary matrices $A \in \Mat_n(\Q[G])$
\begin{equation}\label{Luck_Approx}
    \lim\limits_{i \to \infty} \frac{1}{|G_i|} \dim_{\C}\ker A_i = \dim_G \ker r_A
\end{equation}
where $\dim_G$ denotes the von Neumann dimension.

Jaikin-Zapirain used a completely different approach to extend the L{\"u}ck approximation to matrices over $\C[G]$.
He uses Sylvester matrix rank functions and their natural extensions. 
From this it is possible to show that equation \ref{measure_Equality} holds also for normal matrices over $\C[G].$
In \cite{boschheidgen} the author proved the existence of some function $f: \R_+ \to \R_+$
with $\lim\limits_{\lambda \to 0} f(\lambda) = 0$ such that $\mu_{A_i} (B(0, \lambda)) < f(\lambda)$ also for the case of normal matrices with entries in $\C[G]$. Thus a uniform bound similar as in \ref{uniform_bound} does exist although the function $f$ is not known explicitly.

The problem of convergence of eigenvalues of non normal matrices has already been studied in the theory of random matrices, see \cite{sniady}.

Since for matrices over $\Q[G]$ the methods used by L{\"u}ck in \cite{Luck_start} still work in the case of non normal matrices, a positive answer to question (3) would imply a positive answer to question (1) and (2) for matrices $A \in \Mat_n(\Q[G]).$

We present an example that gives hope to a positive answer to question (3) also in
the case of non normal matrices,
at least for some groups.
Let $G =  \langle g \rangle$ be an infinite cyclic group 
and let 
$$
A = \left(\begin{matrix} g^2 + 3g & 4 \\ 
g^3 & -g^4 + g
\end{matrix}\right).
$$
Let $N_i = \langle g^{5^i}\rangle \subseteq  G $ and therefore 
$G_i \cong \mathbb{Z}/(5^i)\mathbb{Z}$ a cyclic group of order $5^i.$
Note that we get the matrix $A_i$
by replacing $g$ in $A$ by the matrix
\begin{equation}\label{Shift_matrix}
    \left(\begin{matrix}
0 & 1 & 0 & \cdots & 0 \\
\vdots & \ddots & \ddots & & \vdots\\
\vdots & & & \ddots& 0 \\
0 &  &  & \ddots& 1 \\
1 & 0 & \cdots&  & 0
\end{matrix}\right)
\end{equation}
of dimension $\vert G_i \vert = 5^i.$
The following plot shows the eigenvalues of $A_i$ for $i = 1, 2, 3:$
\newline
\includegraphics[width=4.0cm,height=4.0cm]{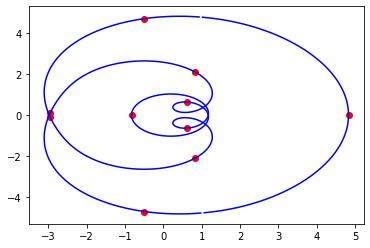}
\includegraphics[width=4.0cm,height=4.0cm]{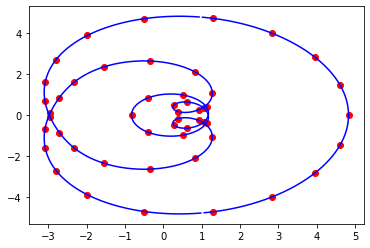}
\includegraphics[width=4.0cm,height=4.0cm]{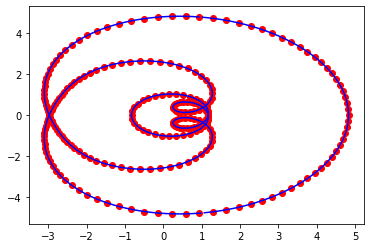}
In the graphics, the limit curve is the support of the Brown measure of the operator $r_A$. 
This will be a consequence of the next theorem.

\begin{theorem}\label{abelian_case}
Let $G$ be a finitely generated abelian group. 
Then the answer to question (1),(2), and (3) is yes.
\end{theorem}
In general, we obtain a negative answer to question (2) and therefore to question (3).
For that we consider the discrete Heisenberg group
$G = H_3(\Z)$, which can be seen as the matrix group
generated by the two matrices
$$
a = \begin{pmatrix}
1 & 1 & 0 \\
0 & 1 & 0 \\
0 & 0 & 1
\end{pmatrix}
\quad \text{and} \quad 
b = \begin{pmatrix}
1 & 0 & 0 \\
0 & 1 & 1 \\
0 & 0 & 1
\end{pmatrix}.
$$ 
Our main result will be the following.
\begin{theorem}
Let $G = H_3(\Z)$ be the Heisenberg group and let $a-b \in \Z[G]$. Let $p$ be an odd prime and set $N_i = \mathrm{Id}_3 + p^i \cdot \mathrm{Mat}_3(\Z) \cap G \unlhd G$ and consider 
the residual chain $G \unrhd N_1 \unrhd N_2 \unrhd \ldots.$ Set $G_i = G / N_i \cong H_3(\Z / p^i \Z).$  Let $A_i \in \text{Mat}_{p^{3i}}(\Z)$
be the matrix that represents the action of $a-b$ on $\C[G_i] \cong \C^{p^{3i}}.$ Let $\mu_{A_i}$ be
the regularized eigenvalue measure of $A_i$.
Then 
\begin{equation}
    \lim\limits_{i \to \infty} \mu_{A_i} (\{0\})
    = \frac{p}{p + 1}.
\end{equation}
\end{theorem}

\subsection*{Acknowledgments} This paper is partially supported by the grants MTM2017-82690-P and PID2020-114032GB-I00 of the Ministry of Science and Innovation of Spain and by the ICMAT Severo Ochoa project CEX2019-000904-S4. 
The author wants to thank Andrei Jaikin-Zapirain and Leo Margolis for helpful discussions and Henrique Souza and an anonymous referee for useful comments on the previous versions of this paper.


\section{Preliminaries}
In this article $G$ will always be 
a residually finite group and $A \in \Mat_n(\C[G])$
will be a matrix over the complex group ring (might be $1 \times 1$).
$N_i \unlhd G \:(i \in \N)$ will be a nested sequence of normal subgroups in $G$ of finite index with trivial intersection.
We will denote the quotients groups by $G_i = G / N_i.$
We will denote by $A_i \in \Mat_{n \cdot |G_i|}(\C)  $ the matrix that represents the action on $\C[G_i]^n$ given by right multiplication with $A$.
Last we will denote by 
$$
    \mu_{A_i} = \frac{1}{\vert G_i \vert}\sum\limits_{j = 1}^{n\cdot \vert G_i \vert} \delta_{\lambda_j}
$$
the regularized eigenvalue measure of $A_i$
and by $\mu_A$ the Brown measure of $r_A$.
\subsection{\texorpdfstring{$\ast$}{}-distributions}

Let $(\NN, \tau)$ be a tracial von Neumann algebra.
We can extend the trace $\tau$ to matrices
$M = (m_{i,j})$ over $\NN$ by
considering $\text{Tr}_{\C}(\tau(m_{i,j}))$.
For an element $a \in \NN$ the $\ast$-distribution of $a$ is the collection of all its $\ast$-moments $(\tau(a^{s_1}a^{s_2} \cdots a^{s_n}))$, where $s_i \in \{1, \ast\}$.
The most important examples are
$(\NN, \tau) = (\Mat_n(\C), \frac{1}{n}\text{Tr})$
and group von Neumann algebras $(\NN(G), \text{Tr}_G)$.
Here we have $\text{Tr}_G(a) = \langle 1_G a, 1_G \rangle$
where $1_G \in \ell^2(G)$ is seen as a vector and 
$a\in \NN(G)$ is seen as an operator acting from the right.

\subsection{Convergence in \texorpdfstring{$\ast$}{}-moments}
We say that a sequence of elements $a_i \in \NN_i$
in different tracial von Neumann algebras $(\NN_i,\tau_i) $
converges in $\ast$-moments to an element $a \in (\NN, \tau)$,
if for all $n$ and $s_1, \ldots, s_n \in \{1, \ast\}$
we have
$$
\lim\limits_{i \to \infty} \tau_i(a_i^{s_1} \cdots a_i^{s_n})
= \tau(a^{s_1} \cdots a^{s_n})
$$
\begin{lemma}
The matrices $A_i$ converge in $\ast$-moments towards the operator $r_A$.
\end{lemma}
\begin{proof}
For simplicity we will only consider the case $n = 1$, that 
means $A = \sum\limits_{g\in G} a_g g \in \C[G]$
with $a_g = 0$ for almost all $g \in G.$
Note that we have
\begin{align*}
\text{Tr}_G (A) &= \langle 1_G A, 1_G \rangle\\
&= \langle \sum\limits_{g \in G} a_G g, 1_G \rangle \\
&=  \sum\limits_{g \in G} a_G \langle g, 1_G \rangle \\
&= a_{1_G}.
\end{align*}
Note further that each $g \in G$ acts  on $\C[G_i] \cong \C^{|G_i|}$ by permutation
of the elements of $G_i \cong G / N_i$. Thus if we have $xg = x$ for some $x \in G_i$ we obtain $xg = x$
for all $x \in G_i$ and therefore $g \in N_i.$
Thus $g \in G$ contributes to the trace of $A_i$
if and only if $g \in N_i$.
Since the intersection of all the $N_i$ is trivial
we get
$$
\lim\limits_{i \to \infty} \frac{1}{|G_i|} \text{Tr}_{\C} A_i = a_{1_G}. 
$$

\end{proof}
\subsection{Brown Measure}\label{Section_Brown_measure}
\subsubsection{Definition and Properties}
For an invertible element $a$ in a tracial von Neumann algebra $(\NN, \tau)$
we define its Fuglede-Kadison determinant as
\begin{equation*}
    \Delta(a) = \exp(\tau(\log |a|) \in (0, \infty),
\end{equation*}
where $|a| = (aa^\ast)^{\frac{1}{2}}.$
For arbitrary elements $a \in \NN$
we set
\begin{equation*}
    \Delta(a) = \lim\limits_{\epsilon \to 0}
    \exp(\tau(\log (aa^\ast + \epsilon)^{\frac{1}{2}}).
\end{equation*}
The function $f(\lambda) = \log(\Delta(a - \lambda))$ is subharmonic (see \cite[Chapter 11]{mingo2017free}). Therefore we can define a measure
$\mu_a$  on the complex numbers by
$$
\int\limits_{\C} \varphi \,d\mu_a = \frac{1}{2\pi}\int\limits_{\R^2} f(\lambda) \nabla^2 \varphi(\lambda) \,d\lambda_r d\lambda_i
$$
for all $\varphi \in C_c^\infty (\R^2)$.
Here $\nabla^2 = \frac{\delta^2}{\delta \lambda_r^2} + \frac{\delta^2}{\delta \lambda_i^2}$ is the Laplace operator, where $\lambda_r$ and $\lambda_i$ are the real and imaginary part of $\lambda = \lambda_r + i\lambda_i \in \C.$
The measure $\mu_a$ is called
the Brown measure associated to $a$.
It is a
 measure on the complex numbers whose support is contained in 
the spectrum of $a$.
For 
$(\NN, \tau) = (\Mat_n(\C), \frac{1}{n} \text{Tr})$
the Brown measure is just the normalized eigenvalue measure.
For normal elements $a$ the Brown measure is just the complex spectral measure $\tau \circ E_a$,
where $E_a$ is the projection valued spectral measure
associated to the normal operator $a$.
For more details see \cite{mingo2017free}.
One important property of the Brown measure is 
that for all $n\in \N$ we have
\begin{equation}\label{polynomial_moments}
\int\limits_{\C} x^n d\mu_a = \tau(a^n).
\end{equation}
In particular we obtain the following result.
\begin{proposition}
For all polynomials $f$ we have
$$
\lim\limits_{i \to \infty} \int\limits_{\C}
f d \mu_{A_i} = \int\limits_{\C}
f d \mu_A
$$
\end{proposition}
\begin{proof}
Since integrals are linear it is enough to show
$$
\lim\limits_{i \to \infty} \int\limits_{\C}
x^n d \mu_{A_i} = \int\limits_{\C}
x^n d \mu_A
$$
for all $n$.
Now the result follows from the convergence in $\ast$-moments of the matrices $A_i$
and \ref{polynomial_moments}.
\end{proof}
Remember that if $a\in \NN$ is normal we have
\begin{equation}\label{mixed_moments}
\int\limits_{\C} x^n \bar{x}^m d\mu_a = \tau(a^n (a^\ast)^m)
\end{equation}
for all $n, m \in \N$. 
Using this and the Stone-Weierstrass approximation Theorem it follows immediately that for normal elements the convergence of $\ast$-moments
implies the weak convergence of spectral measures. However equality \ref{mixed_moments} does not hold for arbitrary elements $a \in \NN.$

\subsubsection{Discontinuity of Brown measure}
Consider an infinite cyclic group $G = \langle g \rangle.$
Since $g\in \NN(G)$ is unitary, especially normal, its $\ast$-distribution
is given by
$$
\text{Tr}_G(g^n (g^{-1})^m) = \begin{cases}
1 \quad \text{if} \quad m = n \\
0 \quad \text{otherwise}
\end{cases}
$$
Let us consider now the sequence $(M_n)$ given by
$$
M_n = 
\begin{pmatrix}
0 & 1 & &\\
& \ddots &\ddots&\\
& & 0 & 1\\
& & & 0
\end{pmatrix} \in \Mat_n(\C).
$$
In $\ast$-moments these matrices converge to $g$ whose Brown/spectral measure is the uniform distribution on the unit circle.
Since the only eigenvalue of $M_n$ is $0$, the eigenvalue measure 
$M_n$ is the Dirac measure at $0$.
However this example is kind of artificial in the sense that we approximate a group element by a sequence of nilpotent matrices. When we use finite quotients, group elements are approximated by permutation matrices.

The problem of the non convergence of Brown measures 
arises from the fact that the eigenvalues of non normal matrices are not stable under small perturbations.
For example, if one changes in the matrices $M_n$ the entry in the lower left corner to a small positive constant $\epsilon,$ their eigenvalue distribution would converge to the uniform distribution on the unit circle, but the limit of the $\ast$-moments would not change at all.
Our counterexample in section (\ref{Heisenberg_Case}) will work in a similar way.

\section{The Abelian Case}
In this section we want to prove Theorem \ref{abelian_case}.
For simplicity we will only consider the case $G \cong \Z$. At the end of the chapter we will explain briefly the necessary changes for the general case.
Let $\{N_m\}$ be a chain of normal subgroups in $G = \langle t \rangle$ with trivial intersection
and set $G_m = G/N_m$.
Each $G_m$ is a cyclic group of finite order. Put
$n_m = |G_m|$ and
consider the $n_m$ dimensional $\C$ vector space $V_m = \C^{n_m} \cong \C[G_m]$.
The action of $G$ on 
$V_m$ is completely determined by the action of $t$ on $V_m$.
Since $t$ acts as a shift on $\C[G_m]$, the action of $t$ on $V_m$ is given by right multiplication with the matrix 
\begin{displaymath}
    T_m = \left(\begin{matrix}
0 & 1 & 0 & \cdots & 0 \\
\vdots & \ddots & \ddots & & \vdots\\
\vdots & & & \ddots& 0 \\
0 &  &  & \ddots& 1 \\
1 & 0 & \cdots&  & 0
\end{matrix}\right).
\end{displaymath}
Thus we obtain algebra homomorphisms
$\rho_m: \C[G] \to \Mat_{n_m}(\C)$ by sending $t$ to $T_m$
and extending linearly.
Obviously we can extend $\rho_m$ to matrices, that means we have
$\rho_m: \Mat_n(\C[G]) \to \Mat_{n \cdot n_m}(\C).$
Let now $A \in \text{Mat}_n(\C[G])$
and $A_m = \rho_m(A)$.
Since $\C[G]$ is abelian, we can define the determinant on
$\Mat_n([\C[G])$. Having a determinant we can also define the characteristic polynomial of $A$
\begin{displaymath}
    p(y) = \det (y\Id_n - A) \in \C[G][y].
\end{displaymath}
By identifying $\C[G]$ with $\C[t^{\pm 1}]$
we obtain a polynomial
$p(t,y) \in \C[t^{\pm 1}, y]$.

\begin{proposition}\label{char_pol_for_abe_groups}
 The characteristic polynomial of $A_m$ is given by 
 $\chi_m(y) = \prod\limits_{\zeta^{n_m} = 1} p(\zeta, y)$
\end{proposition}
For the proof the following lemma will be helpful.
\begin{lemma}\cite[Theorem 1]{Silvester}
Let $R$ be a commutative subring of $\mathrm{Mat}_n(\C)$ and let $M \in \mathrm{Mat}_m(R)$. Then $\det_{\C}(M) = \det_{\C}(\det_R(M)).$
\end{lemma}
\begin{proof}[Proof of the Proposition \ref{char_pol_for_abe_groups}]
  Using the previous lemma, we just have to calculate
  the characteristic polynomial of
  $\rho_m(p(t, y)),$ where $p(t,y) \in \C[G][y]$ is the characteristic polynomial of $A$.
  Since the matrix $T_m$ is conjugated to the diagonal matrix with the $n_m$-th roots of unity on the diagonal the result follows.
\end{proof}
Note that for each $\zeta \in \C$ the polynomial $p(\zeta, y)$ is of fixed degree $n$.
For a polynomial $f\in \C[y]$ with roots $\lambda_1, \ldots, \lambda_n$ counted with multiplicities let us denote by $\mu_f$ the uniform measure on $\{\lambda_1, \ldots, \lambda_n\}.$
With this notation we can write the normalized eigenvalue measure of $A_m$ as
\begin{equation}
    \mu_{A_m} = \frac{1}{n_m} \sum\limits_{\zeta^{n_m}= 1} \mu_{p(\zeta,y)}
\end{equation}

We can define a limit measure $\mu$ by
\begin{equation}\label{abelian_limit_measure}
    \mu = \frac{1}{2\pi} \int\limits_{S^1}\mu_{p(\zeta, y)}d\zeta.
\end{equation}
That means that for a Borel set $K \subseteq \C$
we have 
\begin{displaymath}
    \mu(K) =\frac{1}{2\pi} \int\limits_{S^1} \mu_{p(\zeta, y)}(K) d\zeta.
\end{displaymath}
The following lemma follows directly from \cite[Theorem 5.6]{DYKEMA20163403}.
\begin{lemma}
    The measure $\mu$ from equation \ref{abelian_limit_measure} is exactly the Brown measure $\mu_A$ of the operator $r_A$ on $\ell^2(G)^n$ given by right multiplication with $A$.
\end{lemma}
We now want to show that the measures $\mu_{A_m}$ converge 
weakly towards the measure $\mu$. We will need the following lemma.
\begin{lemma}\cite[Theorem B]{roots_continuous}\label{roots_cont}
Consider a polynomial 
$$
a(x) = x^n + a_1 x^{n-1} + \ldots + a_n = \prod\limits_{i = 1}^s (x - \lambda_i)^{m_i}\in \C[x]
$$
for distinct $\lambda_1, \ldots, \lambda_s.$ Let $\epsilon > 0$ be given such that for $i \neq j$ we have
$B(\lambda_i, \epsilon) \cap B(\lambda_j , \epsilon) = \emptyset.$ Then there exists $\delta > 0$ such that for
$b_i \in B(a_i, \delta)$ for all $i$ the polynomial 
$$
b(x)= x^n + b_1 x^{n-1} + \ldots + b_n 
$$
has exactly $m_j$ roots in $B(\lambda_j, \epsilon)$ for each $j$ where multiplicities are counted.
\end{lemma}

\begin{theorem}
 The measures $\mu_{A_m}$ converge weakly to the Brown measure $\mu_A$.
\end{theorem}
\begin{proof}
For $\zeta \in S^1$ let us denote by $h(\zeta) = (h_1(\zeta), \ldots, h_n(\zeta))$ the roots of $p(\zeta, y)$ counted with multiplicities.
By Lemma \ref{roots_cont}, for each $\zeta$ we can order the roots of $p(\zeta, y)$ in a way such that $h(\zeta)$ as a function 
$h: S^1 \to \C^n$ is continuous.
Let now $f: \C \to \C$ be a continuous function with compact support.
We have
$$
\int\limits_{\C} f d\mu = 
\frac{1}{2\pi}\int\limits_{S^1} \sum_{i = 1}^n f(h_i(\zeta))d\zeta = 
\lim\limits_{m \to \infty} \sum\limits_{\zeta^{n_m} = 1} \sum\limits_{i = 1}^n f(h_i(\zeta)) = 
\lim\limits_{m \to \infty} \int_{\C} f d\mu_{A_m}
$$
since the functions $f(h_i(\zeta))$ are continuous.
\end{proof}

This answers question (3) from the introduction.
The following Theorem answers questions (1) and (2).
\begin{theorem}
    Let $\lambda \in \C$. Then
    \begin{displaymath}
        \lim\limits\limits_{m \to \infty} \mu_{A_m}(\{\lambda\}) = \mu_A (\{\lambda\}).
    \end{displaymath}
\end{theorem}

\begin{proof}
    Note that for polynomials $f_1, f_2 \in \C[y]$ we have
    $\mu_{f_1 \cdot f_2} = \mu_{f_1} + \mu_{f_2}.$
    Remember that we denoted by $p(t, y) \in \C[t^{\pm 1}, y]$ the characteristic polynomial of the matrix $A \in \Mat_n(\C[G]).$
    We can write $p(t, y) = (y - \lambda)^r \cdot p'(t, y)$ where $p'$ has $y$-degree equal to $n - r$ and $(y-\lambda)$ does not divide $p'(t, y).$
    For a polynomial $f \in \C[y]$ and $c \in \C$
    let us denote by $m_f(c)$ the multiplicity with which $c$ appears as a root of $f.$
    
    By definition we have
    \begin{multline*}
        \mu(\{\lambda \}) =  \frac{1}{2\pi} \int\limits_{S^1}\mu_{p(\zeta, y)}(\{\lambda\})d\zeta = 
        \frac{1}{2\pi}\int\limits_{S^1}\mu_{p'(\zeta, y)} + \mu_{(y -\lambda)^r}(\{\lambda\})d\zeta = \\
        \frac{1}{2\pi} \int\limits_{S^1} m_{p'(\zeta, y)}(\lambda) d\zeta + r      
        \leq \frac{1}{2\pi} (n-r)\int\limits_{\substack{\zeta \in S^1\\ p'(\zeta, \lambda) = 0}} 1 \: d\zeta  + r = r
   \end{multline*}
   Here the last equality holds since we are integrating over a finite set which has Lebesgue measure $0.$
   Obviously for each $\zeta \in S^1$ we have
   $\mu_{p(\zeta, y)}(\{\lambda\}) \geq r.$
   By the theorem of Portmanteau we have
   $\limsup\limits_{m \to \infty} \mu_{A_m}(\{\lambda\}) \leq \mu (\{\lambda\}).$
   Putting everything together we obtain
   \begin{displaymath}
       r \leq \limsup\limits_{m \to \infty} \mu_{A_m}(\{\lambda\}) \leq \mu (\{\lambda\}) \leq r.
   \end{displaymath}
   and therefore $\lim\limits_{m \to \infty} \mu_{A_m}(\{\lambda\}) = \mu (\{\lambda\})$
\end{proof}

Let us briefly explain what needs to be done in the case if $G$ is an arbitrary finitely generated abelian group.
If $G \cong \Z^k$ and $A \in \Mat_n(\C[G])$, we can identify
$\C[G]$ with $\C[t_1^{\pm 1}, \ldots, t_k^{\pm 1}]$ 
and obtain a characteristic polynomial $p \in \C[t_1^{\pm 1}, \ldots, t_k^{\pm 1}, y]$ of $A$.
Then for every vector $\zeta = (\zeta_1, \ldots, \zeta_k)$
we obtain a polynomial $p(\zeta_1, \ldots, \zeta_k, y) \in \C[y]$.
To obtain the limit measure as in equation \ref{abelian_limit_measure} we need to integrate over the $k$-dimensional torus $T^k \cong S^1 \times \ldots \times S^1.$
If $G$ is not free abelian but also has a torsion part we 
can write $G \cong \Z^k \times H$ where $H$ is a finite abelian group. Considering our chain $(N_m)_{m \in \N}$ of normal subgroups with trivial intersection there exists $l \in \N$ such that $N_{l} \cap H = \{1\}$ and therefore
$N_l \cong \Z^k.$ Put $L = N_l.$ 
We can consider $\C[G]$ as a right $\C[L]$-module and to the matrix $A \in \Mat_n(\C[G])$ we can associate a matrix 
$\tilde{A} \in \Mat_{n \cdot |G : L|}(\C[L])$ that mirrors the action of $A$ on $\C[G]^n$ seen as a $\C[L]$-module.
We then have $\mu_A = \frac{1}{|G : L|}\mu_{\tilde{A}}$.
Since for $i > l$ we have $N_i \leq N_l$ we also obtain
$\mu_{A_i} = \frac{1}{|G:L|}\mu_{\tilde{A_i}}$.
Thus the case when $G$ has torsion follows from the result for $L$.

\section{A counterexample in the Heisenberg group}\label{Heisenberg_Case}
Let $R = \Z$ or $ = \Z/n\Z$ for some $n \in \N$.
We denote by $G = H_3(R)$ the subgroup of $\text{GL}_3(R)$, generated by matrices 
$a, b$ given by
\begin{equation*}
a = 
    \begin{pmatrix}
    1 & 1 & 0 \\
    0 & 1 & 0 \\
    0 & 0 & 1
    \end{pmatrix}
    \text{ and }
    b = \begin{pmatrix}
    1 & 0 & 0 \\
    0 & 1 & 1 \\
    0 & 0 & 1
    \end{pmatrix}
\end{equation*}
Note that 
\begin{equation*}
    [a, b] = a^{-1}b^{-1}ab = c = \begin{pmatrix}
    1 & 0 & 1 \\
    0 & 1 & 0 \\
    0 & 0 & 1
    \end{pmatrix}
\end{equation*}
and
 \begin{equation*}
    [G , G] = \left\{\begin{pmatrix}
    1 & 0 & \ast \\
    0 & 1 & 0 \\
    0 & 0 & 1
    \end{pmatrix}\right\} = \text{Z}(G).
\end{equation*}
The group $H_3(\Z)$ is called the \textit{discrete Heisenberg group}.

\subsection{Irreducible representations of
\texorpdfstring{$G = H_3(\mathbb{Z} / n\mathbb{Z})$}{}}\label{irred_repr_of_Heis}
In this section we will present the complex irreducible representations of 
$G = H_3(\mathbb{Z} / n\mathbb{Z})$ for $n \in \mathbb{N}$ up to equivalence.
A proof that these are all equivalence classes of irreducible representations can be found in \cite{Grassberger}.
We will describe the equivalence classes of irreducible representations just by giving the images of the matrices $a, b, c$.
Let $k \in \{0, 1, \ldots n-1\}$ and $\omega = \text{e}^{2 \cdot \pi i \cdot \frac{1}{n}}$.
Let $d_1 = \text{gcd}(n, k)$ and $d_2 = \frac{n}{d_1}$.
Then we have $(\omega^k)^{d_2} = 1$ and we get a $d_2$ dimensional irreducible representation 
by mapping
\begin{multline}\label{description_irred_rep}
    c \mapsto \omega^k \cdot  \text{Id}_{d_2}, \quad
    a \mapsto 
    \begin{pmatrix}
    0 & 1 & 0 & \ldots & 0 \\
    0 & 0 & 1 & \ldots & 0 \\
    \vdots & \vdots & \ddots  & \ddots & 0 \\
    0 & 0 &  0 & \ldots & 1 \\
    \omega^{r\cdot d_2} & 0 & 0 & \ldots & 0
    \end{pmatrix}, \\
    b \mapsto \omega^s \cdot
    \begin{pmatrix}
    1 & 0  & \ldots & 0 \\
    0 & \omega^k & \ldots & 0\\
     \vdots&  & \ddots  &  \vdots \\
     0&  \ldots&   0&  \omega^{(d_2-1)\cdot k}
    \end{pmatrix}
\end{multline}
with $r, s \in \{0, 1, \ldots, d_1 -1 \}$.
We will call the irreducible representation with these parameters
$\rho_{k, r, s}.$

\subsection{An almost nilpotent element in \texorpdfstring{$\C[H_3(\mathbb{Z} / n\mathbb{Z})]$}{}}
In this section we want to use the irreducible representations of $G = H_3(\mathbb{Z} / n\mathbb{Z})$ to present an element in the integral group ring $\Z[G]$ that is almost nilpotent. Here by almost nilpotent we mean that the eigenvalue $\lambda = 0$ has a large algebraic multiplicity.
\begin{proposition}
Let $n \in \N$ and
$G = H_3(\mathbb{Z} / n\mathbb{Z}).$ Let  $\rho_{k, r, s}$ be an irreducible 
representation of $G$, using the above notation. Let $d_1 = \gcd(n, k)$ and $d_2 = \frac{n}{d_1}.$
Then 
\begin{equation*}
\rho_{k,r,s}\left((a-b)^n\right) = 
((-1)^{d_2 -1}\omega^{(s-r) \cdot {d_2}} - 1)\cdot \mathrm{Id}_{d_2}
\end{equation*}
\end{proposition}
\begin{proof}
The proof is an explicit calculation. 
In the following  we will use
\begin{itemize}
    \item the equality
    \begin{align*}
    (1-a^{-1}b)a &= a - a^{-1}ba \\
    &= a - bb^{-1}a^{-1}ba\\
    &= a - bc^{-1} \\
    &= a(1 - a^{-1}bc^{-1}),
\end{align*}
\item the fact that $c \in G$ is central and  
\item $G$ has exponent $n$.
\end{itemize}
We have
\begin{align*}
    (a-b)^n & = (a(1 - a^{-1}b))^n \\
    & = a(1 - a^{-1}b)\cdot a(1 - a^{-1}b)\cdot \ldots \cdot a(1 - a^{-1}b) \\
    &= a^n (1 - a^{-1}bc^{-(n-1)})\cdot (1 - a^{-1}bc^{-(n-2)})\cdot \ldots \cdot (1 - a^{-1}b)\\
    & = \prod\limits_{i = 1}^n (1 - a^{-1}b c^{-(n-i)})\\ 
    &= (-1)^n \prod\limits_{i = 1}^n (a^{-1}bc^{-(n-i)} - 1) \\
    &= (-1)^n \prod\limits_{i = 1}^n (a^{-1}bc^i - 1)\\
    &= (-1)^n \prod\limits_{i = 1}^n c^i c^{-i}(a^{-1}bc^i - 1)\\
    &= (-1)^n \prod\limits_{i = 1}^n c^i (a^{-1}b - c^{-i})\\
    & = (-1)^n c^{\frac{n(n+1)}
    {2}}\prod\limits_{i = 1}^n (a^{-1}b - c^i)
\end{align*}
The above calculation took place in 
the group ring $\Z[G]$. We now want to pass to the irreducible representation $\rho_{k,s,r}$
Therefore let $\bar{a} = \rho_{k,s,r}(a), \bar{b} = \rho_{k,s,r}(b)$ and $\bar{c} = \rho_{k,s,r}(c)$.
Note that $\bar{c}$ is now a scalar matrix
with $\bar{c}^{d_2} = \text{Id}.$
We have
\begin{equation*}
    \bar{c}^{\frac{n(n+1)}{2}} = (-1)^{n+1}
    \cdot \text{Id}
\end{equation*}
and therefore we get
\begin{equation}\label{heisenberg_calculation_group_ring}
    (\bar{a}-\bar{b})^n = (-1)^n \bar{c}^{\frac{n(n+1)}
    {2}}\prod\limits_{i = 1}^n (\bar{a}^{-1}\bar{b} - \bar{c}^i)
    = - \prod\limits_{i = 1}^n (\bar{a}^{-1}\bar{b} - \bar{c}^i).
\end{equation}
Here $\bar{c}$ behaves like a primitive $d_2$-th root of unity, that means it commutes with $\bar{a}^{-1}\bar{b}$ and it satisfies $\bar{c}^{d_2} = \text{Id}$ and $\bar{c}^m \neq \text{Id}$ for $m \in \{1, \ldots, d_2 -1$\}. Therefore the last expression in  \ref{heisenberg_calculation_group_ring} is just a cyclotomic polynomial.
That means we get
\begin{align}\label{until_here_for_all_n}
    (\bar{a}-\bar{b})^n &= - \prod\limits_{i = 1}^n (\bar{a}^{-1}\bar{b} - \bar{c}^i) \\
    & = -((\bar{a}^{-1}\bar{b})^{d_2} - 1)^{d_1}
\end{align}
Note that we have
\begin{equation*}
    (\bar{a}^{-1}\bar{b})^{d_2} = (\bar{a}^{-1})^{d_2}\bar{b}^{d_2}\bar{c}^{\frac{d_2(d_2-1)}{2}}.
\end{equation*}
Thus we get
\begin{equation*}
    (\bar{a}-\bar{b})^n = -((\bar{a}^{-1}\bar{b})^{d_2} - 1)^{d_1} = (\bar{a}^{d_2}\bar{b}^{d_2}\bar{c}^{\frac{d_2(d_2-1)}{2}} -1)^{d_1} = 
    ((-1)^{d_2 + 1}\cdot \bar{a}^{-d_2}\bar{b}^{d_2} - 1)^{d_1}
\end{equation*}
In this case an easy calculation shows that
\begin{equation*}
    \bar{a}^{-d_2} = \omega^{-rd_2} \cdot \text{Id}_{d_2} \quad \text{and} \quad
    \bar{b}^{d_2} = \omega^{s \cdot d_2} \cdot \text{Id}_{d_2}.
\end{equation*}
and therefore 
\begin{equation}\label{power_is_scalar}
    (\bar{a}-\bar{b})^n = ((-1)^{d_2 + 1}\cdot\omega^{(s-r) \cdot {d_2}} - 1)\cdot \text{Id}_{d_2}.
\end{equation}
\end{proof}

Note that this is always a scalar matrix. 
That means that the matrix $\rho_{k,r,s}(a-b)$ has either only $0$ as an eigenvalue or only non zero eigenvalues.
We now want to check in which cases this is the zero matrix.
Remember that $\omega$ is a primitive $n$-th root of unity, $n = d_2 \cdot d_1$ and 
$r, s \in \{0, \ldots, d_1 - 1\}$.
That means we have 
$0 \leq (s-r) \cdot {d_2}< n.$
So if $d_2$ is odd then $(a-b)^n = 0$ if and only if
$s = r.$

\subsection{\texorpdfstring{$p$}{p} odd prime}
We are now ready to proof our main theorem.
In this section let $p$ be an odd prime.

\begin{theorem}
Let $G = H_3(\Z)$ be the Heisenberg group and let $a-b \in \Z[G]$. Let $N_i = \mathrm{Id}_3 + p^i \cdot \mathrm{Mat}_3(\Z) \cap G \unlhd G$ and consider 
the residual chain $G \unrhd N_1 \unrhd N_2 \unrhd \ldots.$ Set $G_i = G / N_i \cong H_3(\Z / p^i \Z).$ Note that $\vert G_i \vert = p^{3i}.$ Let $A_i \in \text{Mat}_{p^{3i}}(\Z)$
be the matrix that represents the action of $a-b$ on $\C[G_i] \cong \C^{p^{3i}}.$ Let $\mu_{A_i}$ be
the regularized eigenvalue measure of $A_i$.
Then 
\begin{equation}
    \lim\limits_{i \to \infty} \mu_{A_i} (\{0\})
    = \frac{p}{p + 1}.
\end{equation}
\end{theorem}
\begin{proof}
Let us fix $m \in \N$ and set $n = p^m.$
We want to analyze the regular representation of the group
$G = H_3(\mathbb{Z} / n \mathbb{Z})$. 
We know that the regular representations decomposes as a direct sum
of equivalence classes of irreducible representations, where each irreducible representations appears with the multiplicity of its dimension. We use the notation from the previous section.

For each $i \in \{0, \ldots, m-1\}$ we have
\begin{equation}\label{num_irred_rep}
(p^{m-i} - p^{m-i-1}) \cdot p^i \cdot p^i
\end{equation}
irreducible representations of dimension $d_2 = p^{m-i}$ each one appearing with the multiplicity of $p^{m-i}$. Note that with the notation of section \ref{irred_repr_of_Heis} we have $d_1 = \frac{n}{d_2} = p^i$.
In equation \ref{num_irred_rep} the first factor represents the choices for $k$, to actually get a $d_2$ dimensional representation, and the last two factors represent the choices for $r$ and $s.$
In addition we have $p^{2m}$ one dimensional representations, when the center of $G$ acts trivially, that means $k = 0.$
We now want to count how often the eigenvalue $0$ appears in the regular representations of the element $a - b \in \mathbb{Z}[G]$.
For that have us let a look on equation \ref{power_is_scalar}.
This gives us that the element $a-b$ is nilpotent in the irreducible representations if and only if $s = r.$
In this case the only eigenvalue of $a-b$ is $0$. If $s\neq r$ then $a-b$ is a root of some non zero scalar matrix and has therefore no eigenvalue equal to $0$.
Therefore we just have to count all the cases
when $s = r$ and add up their total dimensions.
Let us sum up the multiplicity of the eigenvalue $0$ in the regular representation.
Let us decompose $\C[G]$ as $\C[G] \cong M_1 \oplus M_2$, where $M_2$ is the direct sum of all 
one dimensional irreducible representations of $G$.
The following sum describes the algebraic multiplicity of $0$ in the representation of $a-b$ on $M_1$.
\begin{equation*}
    S_1 = \sum\limits_{i = 0}^{m-1} p^{m-i} \cdot p^{m-i} \cdot (p^{m-i}-p^{m-i-1}) \cdot p^i
    = p^{3m} (1-\frac{1}{p}) \sum\limits_{i = 0}^{m-1} \left(\frac{1}{p^2}\right)^i
\end{equation*}
Let us explain this equation.
The first $p^{m-i}$ is the dimension of the irreducible representation. The second $p^{m-i}$
is its multiplicity with which it appears in the regular representation. The third factor is the number of choices we have for $k$ to get a $p^{m-i}$ dimensional irreducible representation.
The last factor represents the choices of $r$ and $s$ such that $r = s$.
Further we have
 $S_2 = p^m$ times the eigenvalue $0$
in the one dimensional representations in $M_2$, that means when $k = 0, r = s.$

We are now interested in the normalized multiplicity of the eigenvalue $0$, that means in $\mu_{A_m}(\{0\}) = \frac{S_1 + S_2}{p^{3m}}$.
Using the above equations we get
\begin{equation*}
    \mu_{A_m} (\{0\}) = 
    (1- \frac{1}{p}) \cdot
    \left(\sum\limits_{i = 0}^{m-1} \left(\frac{1}{p^2}\right)^i\right)
    + \frac{1}{p^{2m}}.
\end{equation*}
We are interested in the limit behaviour, that means in 
$\lim\limits_{m \to \infty} \mu_{A_m}(\{0\}).$
Using the geometric series we get 
\begin{align*}
    \lim\limits_{m \to \infty} \mu_{A_m}(\{0\}) &
    = \lim\limits_{m \to \infty}
    (1- \frac{1}{p}) \cdot
    \sum\limits_{i = 0}^{m-1} \left(\frac{1}{p^2}\right)^i \\  
    &= (1- \frac{1}{p}) \cdot \frac{p^2}{p^2 -1} \\
    &= \frac{p-1}{p} \cdot \frac{p^2}{(p+1)(p-1)}\\
    &= \frac{p}{p+1}
\end{align*}
\end{proof}

\begin{remark}
Let $A = a-b \in \Z[G]$.
From \cite[Theorem 1.4]{schick} we know that
$$
\int\limits_0^1 \log t d \mu_{| A |} > - \infty.
$$
It follows then from \cite[Proposition 2.16]{brown_measure_unbound} 
that for the Brown measure $\mu_A$ of $A$ we have
$$
\mu_A(\{0\}) = 0.
$$
\end{remark}

\printbibliography

@article{Grassberger,
    author    = "Grassberger,J. and  Hörmann, G.",
    title     = "A note on representations of the finite Heisenberg group and sums of greatest common divisors",
    journal   = "Discrete Mathematics and Theoretical Computer Science",
    ?_volume   = "4",
    ?_number   = "(2)",
    ?_pages    = "pp.91-100",
    year      = "2001",
    ?_month    = "",
    ?_note     = "",
}

@article{Jaikin1,
author = {Jaikin-Zapirain, A.},
year = {2019},
pages = {},
title = {The base change in the Atiyah and the Lück approximation conjectures},
volume = {29},
journal = {Geometric and Functional Analysis}
}

@article{Luck_start,
	author = {L{\"u}ck, W. },
	da = {1994/07/01},
	id = {L{\"u}ck1994},
	journal = {Geometric \& Functional Analysis GAFA},
	number = {4},
	pages = {455--481},
	title = {Approximating $L^2$-invariants by their finite-dimensional analogues},
	ty = {JOUR},
	volume = {4},
	year = {1994}
}

@article{Silvester,
 author = {J. R. Silvester},
 journal = {The Mathematical Gazette},
 number = {501},
 pages = {460--467},
 publisher = {Mathematical Association},
 title = {Determinants of Block Matrices},
 volume = {84},
 year = {2000}
}

@book{mingo2017free,
  title={Free Probability and Random Matrices},
  author={Mingo, J.A. and Speicher, R.},
  isbn={9781493969425},
  series={Fields Institute Monographs},
  year={2017},
  publisher={Springer New York}
}

@article{DYKEMA20163403,
title = {On reduction theory and Brown measure for closed unbounded operators},
journal = {Journal of Functional Analysis},
volume = {271},
number = {12},
pages = {3403-3422},
year = {2016},
author = {K. Dykema and J. Noles and F. Sukochev and D. Zanin}
}

@article{roots_continuous,
author = {Harris, G. and Martin, C.},
year = {1987},
pages = {390-392},
title = {The roots of a polynomial vary continuously as a function of the coefficients},
volume = {100},
journal = {Proceedings of the American Mathematical Society}
}

@article{sniady,
author = {Sniady, P.},
journal = {J. Funct. Anal},
year = {2002},
volume = {193},
number = {2},
pages = {291--313},
month = {},
pages = {},
title = {Random regularization of Brown spectral measure}
}

@article{schick,
 author = {T. Schick},
 journal = {Transactions of the American Mathematical Society},
 number = {8},
 pages = {3247--3265},
 publisher = {American Mathematical Society},
 title = {$L^2$-Determinant Class and Approximation of $L^2$-Betti Numbers},
 volume = {353},
 year = {2001}
}

@article{brown_measure_unbound,
 author = {U. Haagerup and H. Schultz},
 journal = {Mathematica Scandinavica},
 number = {2},
 pages = {209--263},
 publisher = {Mathematica Scandinavica},
 title = {Brown measures of unbounded operators affiliated with a finite von Neumann algebra},
 volume = {100},
 year = {2007}
}

@manual{boschheidgen,
author = {Boschheidgen, J.},
title ={PhD-Thesis, in preparation},
year = {2023},
publisher = {Universidad Autonoma de Madrid}
}
\end{document}